\documentclass[10pt]{article}

\usepackage{amsmath,amssymb,amsthm,bm}
\usepackage{geometry}
\usepackage{booktabs}
\usepackage{enumitem}
\usepackage{authblk}
\usepackage[colorlinks=true,linkcolor=blue,citecolor=blue,urlcolor=blue]{hyperref}

\newtheorem{theorem}{Theorem}[section]
\newtheorem{lemma}[theorem]{Lemma}

\newtheorem{corollary}[theorem]{Corollary}
\theoremstyle{remark}
\newtheorem{remark}[theorem]{Remark}
\newtheorem{conjecture}[theorem]{\bf Conjecture}
\usepackage{xcolor}

\newcommand{\R}{\mathbb{R}}
\newcommand{\T}{\mathsf{T}}
\newcommand{\rank}{\operatorname{rank}}
\newcommand{\diag}{\operatorname{diag}}
\allowdisplaybreaks[4]
\title{A proof of the interlacing conjecture for the singular values of Sylvester–Kac matrices}
\author[1]{Fuying Tang}
\author[2]{Shuo Shi\thanks{Corresponding author. E-mail: tangfy@shu.edu.cn (Fuying Tang); shishuomath@foxmail.com (Shuo Shi)}}
\affil[1]{Department of Mathematics, Shanghai University, Shanghai, 200444, People’s Republic of China}
\affil[2]{School of Mathematics and Computational Science, Xiangtan University, Xiangtan, Hunan,  411105, People’s Republic of China}
\date{}

\begin{document}

		\maketitle

\begin{abstract}
Boros and R\'ozsa (Linear Algebra Appl. 421 (2007) 407--416) reduced the singular value problem for the odd-order Sylvester--Kac matrix $P\in\mathbb{R}^{(2m+1)\times(2m+1)}$ to those for $Q^{(m)}\in\mathbb{R}^{(m+1)\times m}$ and 	$S^{(m)}\in\mathbb{R}^{m\times(m+1)}$ via a permutation similarity transformation. 
They noted that $(Q^{(m)})^{\T}Q^{(m)}-S^{(m)}(S^{(m)})^{\T}$ is positive definite, and also conjectured that $\sigma_i(S^{(m)})\ge\sigma_{i+1}(Q^{(m)})$ for $1\le i\le m-1$ and $m\ge2$.
In this paper, we prove that these inequalities are strict and establish the strict interlacing relation 	
\begin{equation*}
\sigma_1(Q^{(m)})>\sigma_1(S^{(m)})>
\cdots>\sigma_m(Q^{(m)})>
\sigma_m(S^{(m)})>0
\end{equation*}
for all $m\ge1$.
Here, $\sigma_i(\cdot)$ denotes the $i$-th largest singular value of a matrix.
~
\end{abstract}

\section{Introduction}
	
Let $\R$ denote the field of real numbers, and let $\R^{n\times m}$ denote the set of all $n\times m$ real matrices. 
For a real matrix $A=(A_{ij})\in\R^{n\times m}$, $A_{ij}$ denotes the entry in the $i$-th row and $j$-th column of $A$. 
$A^{\T}$ and $\rank A$ denote its transpose and rank, respectively. 
When $n=m$, $\operatorname{spec}(A)$ denotes its spectrum. 
We write $I_r$ for the $r\times r$ identity matrix and $\diag(a_1,\ldots,a_r)$ for the diagonal matrix with diagonal entries $a_1,\ldots,a_r$. 
The eigenvalues of a real symmetric matrix 	$A\in\mathbb R^{n\times n}$ are arranged in nonincreasing order as 	$\lambda_1(A)\ge\lambda_2(A)\ge\cdots\ge\lambda_n(A)$.
For a matrix $A\in\mathbb R^{n\times m}$, let 	$k=\min\{n,m\}$.
Its singular values are arranged in nonincreasing order as 	$\sigma_1(A)\ge\sigma_2(A)\ge\cdots\ge\sigma_k(A)\ge0$.
For a real symmetric matrix $A\in \R^{n\times n}$, we write $A\succeq0$ and $A\succ0$ to mean that $A$ is positive semidefinite and positive definite, respectively.
	
The Sylvester--Kac matrix is a classical tridiagonal matrix whose spectral properties have been studied for more than a century.
The Sylvester--Kac matrix originates in Sylvester's work on determinants and was later studied by Kac.
Its eigenvalues and eigenvectors admit explicit descriptions~\cite{Sylvester1854,Kac1947,ChuKilic2024}. 
Subsequent studies have extended the spectral theory of Sylvester--Kac type matrices in both structural and applied directions. 
Bevilacqua and Bozzo identified a four-dimensional space of tridiagonal matrices containing the classical Sylvester--Kac matrix and admitting simultaneous 
triangularization~\cite{BevilacquaBozzo2009}, while da Fonseca et al. derived a Sylvester--Kac type matrix from a deposition and evaporation model on a finite array and used its eigenpairs to describe the model's time-dependent solution~\cite{daFonseca2013}. 
More recent studies have produced a variety of explicitly solvable Sylvester--Kac-type matrices, including families with square-integer or quadratic spectra, structured Jacobi realizations, and periodic perturbations \cite{VaiaSpadini2020,daFonsecaKilic2021,DuDaFonseca2024,DuDaFonseca2025,EncinasEtAl2025}. 
Thus, the eigenvalue structure of the Sylvester--Kac matrix and many of its variants is now relatively well understood. 
By contrast, the corresponding singular value problem is substantially less transparent. 
In particular, the standard Sylvester--Kac matrix is generally nonnormal, so its singular values cannot be recovered directly from its explicitly known eigenvalues.
	
For a positive integer $m$, the $(2m+1)\times(2m+1)$ Sylvester--Kac matrix $P\in\R^{(2m+1)\times(2m+1)}$ is given by
\begin{equation}\label{eqSKP}
		P=
		\begin{bmatrix}
			0&2m&&&&0\\
			1&0&2m-1&&&\\
			&2&0&2m-2&&\\
			&&\ddots&\ddots&\ddots&\\
			&&&2m-1&0&1\\
			0&&&&2m&0
		\end{bmatrix}.
\end{equation}
Equivalently,
\begin{equation*}
		P_{ij}=
		\begin{cases}
			2m-i+1, & i-j=-1,\\
			j,      & i-j=1,\\
			0,      & \text{otherwise}.
		\end{cases}
\end{equation*}
Boros and R\'ozsa~\cite{BorosRozsa2007} investigated the singular values of the odd-order Sylvester--Kac matrix $P$ defined above.
Using a permutation similarity transformation, they reduced the singular value problem for $P$ to those for two rectangular matrices $Q^{(m)}$ and $S^{(m)}$.
Specifically, let $e_i$ denote the $i$-th standard basis vector of $\R^{2m+1}$, and define the permutation matrix
\begin{equation*}
 	\Pi=[e_1,e_3,\ldots,e_{2m+1},e_2,e_4,\ldots,e_{2m}].
\end{equation*}
Then
\begin{equation}
 	\Pi^{\T}P\Pi=
 	\begin{bmatrix}
 		0&Q^{(m)}\\
 		S^{(m)}&0
 	\end{bmatrix},
 	\label{eq:block-decomposition}
\end{equation}
where
\begin{equation}
 	\begin{aligned}
 		Q^{(m)}
 		&=2
 		\begin{bmatrix}
 			m&&&&0\\
 			1&m-1&&&\\
 			&2&m-2&&\\
 			&&\ddots&\ddots&\\
 			&&&m-1&1\\
 			0&&&&m
 		\end{bmatrix}
 		\in\R^{(m+1)\times m},
 	\end{aligned}
 	\label{eq:Q-definition}
\end{equation}
equivalently,
\begin{equation*}
 	\begin{aligned}
 		Q^{(m)}_{ij}
 		&=\begin{cases}
 			2(m-i+1), & i-j=0,\\
 			2j,       & i-j=1,\\
 			0,        & \text{otherwise},
 		\end{cases}
 	\end{aligned}
\end{equation*}
and
\begin{equation}
 	\begin{aligned}
 		S^{(m)}
 		&=
 		\begin{bmatrix}
 			1&2m-1&&&0\\
 			&3&2m-3&&\\
 			&&\ddots&\ddots&\\
 			0&&&2m-1&1
 		\end{bmatrix}
 		\in\R^{m\times(m+1)},
 	\end{aligned}
 	\label{eq:S-definition}
\end{equation}
equivalently,
\begin{equation*}
 	\begin{aligned}
 		S^{(m)}_{ij}
 		&=\begin{cases}
 			2i-1,       & i-j=0,\\
 			2m-2i+1,    & i-j=-1,\\
 			0,          & \text{otherwise}.
 		\end{cases}
 	\end{aligned}
\end{equation*}
Since $\Pi$ is orthogonal, $P$ and $\Pi^{\T}P\Pi$  have the same singular values.
Moreover, \eqref{eq:block-decomposition} gives
\begin{equation}
 	(\Pi^{\T}P\Pi)^{\T}(\Pi^{\T}P\Pi)
 	=
 	\begin{bmatrix}
 		(S^{(m)})^{\T}S^{(m)}&0\\
 		0&(Q^{(m)})^{\T}Q^{(m)}
 	\end{bmatrix}.
 	\label{eq:block-gram}
\end{equation}
The matrix $Q^{(m)}$ has full column rank, and  $S^{(m)}$ has full row rank.
Consequently, as observed by Boros and R\'ozsa~\cite{BorosRozsa2007}, the singular values of $P$ consist of the $m$ positive singular values of $Q^{(m)}$, the $m$ positive singular values of $S^{(m)}$, and one zero singular value.
Using generating functions and a second-order differential equation, they further obtained an explicit formula for the singular values of $Q^{(m)}$.
Boros and R\'ozsa~\cite{BorosRozsa2007} also observed that 
$$(Q^{(m)})^{\T}Q^{(m)}-S^{(m)}(S^{(m)})^{\T}$$
is positive definite. 
This implies that
\begin{equation}\label{lower order}
	\sigma_i(Q^{(m)})>\sigma_i(S^{(m)}),~1\leq i\leq m,
\end{equation}
while \cite{BorosRozsa2007} states the non-strict form.
They further conjectured that the following result also holds.

	\begin{conjecture} \rm \cite{BorosRozsa2007}
		\label{conj:boros-interlacing}
		Let $Q^{(m)}$ and $S^{(m)}$ be defined by~\eqref{eq:Q-definition} and \eqref{eq:S-definition}, respectively. Then for any integer $m\geq 2$,
		\begin{equation}\label{conjecture inequalities}
			\sigma_i(S^{(m)})\geq\sigma_{i+1}(Q^{(m)}),~
			1\leq i\leq m-1.
		\end{equation}
	\end{conjecture}
	
However, this conjecture has remained open.
The remainder of the paper is organized as follows.
In Section~\ref{Mainresults}, we show that the inequalities
in \eqref{conjecture inequalities} hold strictly, thereby proving
Conjecture~\ref{conj:boros-interlacing}.
	
\section{Main results}
\label{Mainresults}
	
In this section, we prove that Conjecture~\ref{conj:boros-interlacing} holds.
First, we introduce some necessary tools and results.
Using the standard relation between eigenvalues and singular values, we have
\begin{equation}
		\lambda_i(\mathcal G^{(m)})=\sigma_i(Q^{(m)})^2,~
		\lambda_i(\mathcal H^{(m)})=\sigma_i(S^{(m)})^2,~1\leq i\leq m,
		\label{eq:eigen-singular-order}
\end{equation}
where
\begin{equation*}
		\mathcal G^{(m)}=(Q^{(m)})^{\T}Q^{(m)},~
		\mathcal H^{(m)}=S^{(m)}(S^{(m)})^{\T}.
\end{equation*}
In fact, both matrices are positive definite.
Thus, to prove that the inequalities in \eqref{conjecture inequalities} hold strictly, it suffices to show that
\begin{equation*}
		\begin{aligned}
			\lambda_i\bigl(\mathcal H^{(m)}\bigr)
			>
			\lambda_{i+1}\bigl(\mathcal G^{(m)}\bigr),~1\leq i\leq m-1,~m\geq 2.
		\end{aligned}
\end{equation*}
By direct calculation, the entries of the symmetric tridiagonal matrices $\mathcal G^{(m)}=(\mathcal G^{(m)}_{ij})\in\mathbb{R}^{m\times m}$ and $\mathcal H^{(m)}=(\mathcal H^{(m)}_{ij})\in\mathbb{R}^{m\times m}$ are given by
\begin{equation}
			\mathcal G^{(m)}_{ij}=
			\begin{cases}
				4\bigl((m-i+1)^2+i^2\bigr), & i=j,\\
				4i(m-i), & j=i+1,\\
				4j(m-j), & i=j+1,\\
				0, & \text{otherwise},
			\end{cases}
			\label{eq:G-entries}
\end{equation}
and
\begin{equation}
			\mathcal H^{(m)}_{ij}=
			\begin{cases}
				(2i-1)^2+(2m-2i+1)^2, & i=j,\\
				(2i+1)(2m-2i+1),      & j=i+1,\\
				(2j+1)(2m-2j+1),      & i=j+1,\\
				0,                     & \text{otherwise}.
			\end{cases}
			\label{eq:H-entries}
\end{equation}
	
The following two standard results will be used to prove Conjecture~\ref{conj:boros-interlacing}.
	
\begin{lemma}\rm\cite{Horn2012} 
\label{lem:loewner-monotonicity}
Let $A,B\in\R^{n\times n}$ be real symmetric matrices.
If $A-B\succeq0$, then
		\begin{equation*}
			\lambda_i(A)\geq\lambda_i(B),~1\leq i\leq n.
\end{equation*}
\end{lemma}
	
\begin{lemma}\rm\cite{Horn2012} 
\label{lem:rectangular-gram}
Let $A\in\R^{n\times m}$ and $B\in\R^{m\times n}$ with $n\le m$.  
Then the $m$ eigenvalues of $BA$ consist of the $n$ eigenvalues of $AB$ with the same algebraic multiplicities, together with $m-n$ additional zeros.
\end{lemma}

The following singular value formula was obtained by
Boros and R\'ozsa~\cite{BorosRozsa2007}.
We also record the relations needed below.

\begin{theorem}\label{Boros--Rozsa formulaa}
Let $Q^{(m)}\in\mathbb R^{(m+1)\times m}$ be defined by \eqref{eq:Q-definition}.
Then for any integer $m\ge1$,
\begin{equation}\label{eq:Q-singular-values}
    \sigma_i(Q^{(m)})
    =2\sqrt{(m-i+1)(m+i)},~
1\le i\le m.
\end{equation}
Consequently,
\begin{align}
			\lambda_i(\mathcal G^{(m)})
			&=8m+\lambda_i(\mathcal G^{(m-1)}),~1\leq i\leq m-1,~m\ge2
			\label{eq:G-eigenvalue-recurrence}\\
			\lambda_m(\mathcal G^{(m)})&=8m,~m\ge1.
			\label{eq:G-smallest-eigenvalue}
		\end{align}
\end{theorem}

In fact, Theorem \ref{Boros--Rozsa formulaa} gives a recurrence relation among the eigenvalues of $\mathcal G^{(m)}$; 
hence we hope for a similar result for $\mathcal H^{(m)}$. 
Next, we carry out an explicit decomposition of $\mathcal H^{(m)}$.
 	
\begin{lemma} 
\label{lem:H-reductions}
Let $\mathcal H^{(m)}$ and $\mathcal H^{(m-1)}$ be defined by~\eqref{eq:H-entries}. 
Then for any integer $m\geq2$,
\begin{align}
			&\mathcal H^{(m)}
			=(L^{(m)})^{\T}L^{(m)}
			+\diag(1,0,\ldots,0,1),                             \label{eq:H-first-reduction}\\
			&L^{(m)}(L^{(m)})^{\T}
			=\mathcal H^{(m-1)}+8mI_{m-1},                    \label{eq:H-second-reduction}
\end{align}
where $L^{(m)}=(L^{(m)}_{ij}) \in \R^{(m-1)\times m}$ is the upper bidiagonal matrix defined by
\begin{equation}\label{explicit-form-L}
			L^{(m)}_{ij}=
			\begin{cases}
				2m-2i+1, & i=j,\\
				2i+1,    & j=i+1,\\
				0,       & \text{otherwise}.
			\end{cases}
\end{equation}
\end{lemma}
	
\begin{proof} 
Let $m\ge2$.
For $2\le j\le m-1$, the $j$-th column of $L^{(m)}$	has exactly two nonzero entries:
\begin{equation*}
			L^{(m)}_{j-1,j}=2j-1,~L^{(m)}_{jj}=2m-2j+1.
\end{equation*}
Hence,
\begin{equation*}
		\bigl((L^{(m)})^{\mathrm T}L^{(m)}\bigr)_{jj}
		=
		(2j-1)^2+(2m-2j+1)^2
		=
		\mathcal H^{(m)}_{jj},
		\qquad 2\le j\le m-1.
\end{equation*}
The first and last columns of $L^{(m)}$ each have a single nonzero entry, equal to $2m-1$.
Hence, the two endpoint diagonal entries of $(L^{(m)})^{\T}L^{(m)}$ are both $(2m-1)^2$,
whereas those of $\mathcal H^{(m)}$ are both $(2m-1)^2+1$. 
Also,
\begin{equation*}
		\bigl((L^{(m)})^{\mathrm T}L^{(m)}\bigr)_{i(i+1)}
		=
		(2m-2i+1)(2i+1)
		=
		\mathcal H^{(m)}_{i(i+1)},
		\qquad 1\le i\le m-1.
\end{equation*}
Since both matrices are symmetric tridiagonal, \eqref{eq:H-first-reduction} follows.
To prove \eqref{eq:H-second-reduction}, we first compare the diagonal entries.
For $1\le i\le m-1$, direct calculation gives
\begin{equation*}
		\begin{aligned}
			\bigl(L^{(m)}(L^{(m)})^{\T}\bigr)_{ii}
			&=(2m-2i+1)^2+(2i+1)^2\\
			&=(2m-2i-1)^2+(2i-1)^2+8m\\
			&=\mathcal H^{(m-1)}_{ii}+8m.
		\end{aligned}
\end{equation*}
For $1\le i\le m-2$,
\begin{equation*}
	\bigl(L^{(m)}(L^{(m)})^{\mathrm T}\bigr)_{i(i+1)}
	=
	(2i+1)(2m-2i-1)
	=
	\mathcal H^{(m-1)}_{i(i+1)}.
\end{equation*}
Since $L^{(m)}(L^{(m)})^{\T}$ and $\mathcal H^{(m-1)}$ are symmetric tridiagonal, \eqref{eq:H-second-reduction} follows.
\end{proof}
	
By Lemma \ref{lem:H-reductions}, we obtain a lower bound for the eigenvalues of $\mathcal H^{(m)}$ in terms of those of $\mathcal H^{(m-1)}$.

\begin{lemma}
\label{lem:H-eigenvalue-lower-bound}
Let $\mathcal H^{(m)}$ and $\mathcal H^{(m-1)}$ be defined by~\eqref{eq:H-entries}. 
Then for any integer $m\geq2$,
\begin{equation}
		\lambda_i(\mathcal H^{(m)})
		\ge 8m+\lambda_i(\mathcal H^{(m-1)}),~
		1\le i\le m-1.
		\label{eq:H-eigenvalue-recurrence-bound}
\end{equation}
\end{lemma}

\begin{proof}
Let $m\ge2$, and let $L^{(m)}$ be the matrix defined in \eqref{explicit-form-L}.
The first $m-1$ columns of $L^{(m)}$ form an upper triangular matrix with positive diagonal entries $2m-2i+1$, $1\le i\le m-1$.
Hence, $\rank L^{(m)}=m-1$.
Thus, $L^{(m)}(L^{(m)})^{\T}$ is positive definite, while $(L^{(m)})^{\T}L^{(m)}$ has $m-1$ positive eigenvalues and one zero eigenvalue.
By \eqref{eq:H-second-reduction} and Lemma~\ref{lem:rectangular-gram}, the eigenvalues of $(L^{(m)})^{\T}L^{(m)}$ satisfy
\begin{equation*}
		\begin{aligned}
			\lambda_i\bigl((L^{(m)})^{\T}L^{(m)}\bigr)
			&=8m+\lambda_i(\mathcal H^{(m-1)}),~1\le i\le m-1,\\
			\lambda_m\bigl((L^{(m)})^{\T}L^{(m)}\bigr)
			&=0.
		\end{aligned}
\end{equation*}
Moreover, \eqref{eq:H-first-reduction} gives
\begin{equation*}
		\mathcal H^{(m)}-(L^{(m)})^{\T}L^{(m)}
		=\diag(1,0,\ldots,0,1)
		\succeq0.
\end{equation*}
Therefore, Lemma~\ref{lem:loewner-monotonicity} yields
\begin{equation*}
		\begin{aligned}
			\lambda_i(\mathcal H^{(m)})
			&\ge
			\lambda_i\bigl((L^{(m)})^{\T}L^{(m)}\bigr)\\
			&=8m+\lambda_i(\mathcal H^{(m-1)}),
			\qquad 1\le i\le m-1.
		\end{aligned}
\end{equation*}
\end{proof}
	
From the relation between the eigenvalues of $\mathcal H^{(m)}$ and those of $\mathcal H^{(m-1)}$, we obtain the following result.
	
\begin{theorem}
		\label{thm:strict-interlacing}
Let $Q^{(m)}$ and $S^{(m)}$ be defined by~\eqref{eq:Q-definition} and \eqref{eq:S-definition}, respectively. 
Then for any integer $m\geq 2$
\begin{equation}	\label{eq:singular-interlacing}
			\sigma_i(S^{(m)})>\sigma_{i+1}(Q^{(m)}), 1\leq i\leq m-1.
\end{equation}
\end{theorem}
	
\begin{proof}
By \eqref{eq:eigen-singular-order}, it suffices to prove
\begin{equation}
			\begin{aligned}
				&\lambda_i(\mathcal H^{(m)})
				>\lambda_{i+1}(\mathcal G^{(m)}), 1\leq i\leq m-1.
			\end{aligned}
			\label{eq:eigenvalue-interlacing}
\end{equation}
We prove this inequality by induction on $m$.
For $m=2$, 
\begin{equation*}
			\mathcal G^{(2)}=4\begin{bmatrix}
				5&1\\
				1&5
			\end{bmatrix},~
			\mathcal H^{(2)}=\begin{bmatrix}
				10&9\\
				9&10
			\end{bmatrix}.
\end{equation*}
Thus, 
\begin{equation*}
			\lambda_1(\mathcal H^{(2)})=19>
			\lambda_2(\mathcal G^{(2)})=16,
\end{equation*} 
which proves \eqref{eq:eigenvalue-interlacing} for $m=2$.
Let $m\ge3$ and assume that \eqref{eq:eigenvalue-interlacing} holds for $m-1$.
For $1\le i\le m-2$, the induction hypothesis gives
\begin{equation}
			\lambda_i(\mathcal H^{(m-1)})
			>\lambda_{i+1}(\mathcal G^{(m-1)}).
			\label{eq:induction-used}
\end{equation}
By Lemma~\ref{lem:H-eigenvalue-lower-bound}, \eqref{eq:induction-used}, and \eqref{eq:G-eigenvalue-recurrence}, we obtain
\begin{equation*}
		\begin{aligned}
			\lambda_i(\mathcal H^{(m)})
			&\ge
			8m+\lambda_i(\mathcal H^{(m-1)})\\
			&>
			8m+\lambda_{i+1}(\mathcal G^{(m-1)})\\
			&=
			\lambda_{i+1}(\mathcal G^{(m)}),~1\le i\le m-2.
		\end{aligned}
\end{equation*}
For $i=m-1$, the positive definiteness of $\mathcal H^{(m-1)}$ gives $\lambda_{m-1}(\mathcal H^{(m-1)})>0$. 
By Lemma~\ref{lem:H-eigenvalue-lower-bound} and \eqref{eq:G-smallest-eigenvalue},
\begin{equation*}
			\lambda_{m-1}(\mathcal H^{(m)})
			\geq8m+\lambda_{m-1}(\mathcal H^{(m-1)})
			>8m
			=\lambda_m(\mathcal G^{(m)}).
\end{equation*}
Thus, \eqref{eq:eigenvalue-interlacing} holds for $1\le i\le m-1$, which completes the induction.
Since $\mathcal H^{(m)}$ and $\mathcal G^{(m)}$ are positive definite, taking positive square roots in \eqref{eq:eigenvalue-interlacing} gives \eqref{eq:singular-interlacing}.
\end{proof}
	
\begin{remark}\label{c}
Since \eqref{eq:singular-interlacing} holds, the inequalities in \eqref{conjecture inequalities} follow.
Hence, Conjecture~\ref{conj:boros-interlacing} holds.
\end{remark}
	
Combining \eqref{lower order} and \eqref{eq:singular-interlacing} yields the following corollary for $m\ge2$. The case $m=1$ follows directly from \eqref{lower order}.
	
\begin{corollary}
\label{cor:strict-interlacing}
Let $Q^{(m)}$ and $S^{(m)}$ be defined by \eqref{eq:Q-definition} and \eqref{eq:S-definition}, respectively. 
Then
\begin{equation*}
\sigma_1(Q^{(m)})>\sigma_1(S^{(m)})>0
\end{equation*}
for $m=1$, while
\begin{equation*}
\sigma_1(Q^{(m)})>\sigma_1(S^{(m)})>
			\sigma_2(Q^{(m)})>\sigma_2(S^{(m)})>
			\cdots>\sigma_m(Q^{(m)})>
			\sigma_m(S^{(m)})>0
\end{equation*}
for $m\geq 2$.
\end{corollary}

\section*{AI statement}

The paper was written jointly by the two authors. ChatGPT 6 Astra was used only for language polishing and for screening counterexamples. 
The authors take responsibility for all results in the paper.

\end{document}